\documentclass{amsart}
\usepackage{amsthm, amsmath,amsfonts, amssymb, enumerate}
\usepackage{graphicx,tikz}
\usepackage{enumitem}

\usepackage[pdftex,pdfpagelabels,bookmarks,hyperindex,hyperfigures]{hyperref}

\newcommand{\inv}[1]{#1^{-1}}

\newcommand{\conc}{\text{Condition } \mathrm{(C)}}
\newcommand{\core}{C}
\newcommand{\conch}[1]{\text{Condition (C)}_{#1}}
\newcommand{\czero}{\mathrm{CAT}(0)}

\newcommand{\h}{\mathcal H}
\newcommand{\Isom}{\text{Isom}}
\newcommand{\Stab}{\text{Stab}}
\newcommand{\comm}[2]{N_{#1}[#2]}

\newcommand{\rk}[1]{\mathrm{rk}(#1)}
\newcommand{\acts}{\curvearrowright}
\newcommand{\minset}[1]{\mathrm{Min}(#1)}
\newcommand{\pione}[1]{\pi_1(#1)}
\newcommand{\Gam}{G}

\numberwithin{equation}{section}
\theoremstyle{plain}

\newtheorem{lem}[equation]{Lemma}
\newtheorem{prop}[equation]{Proposition}
\newtheorem{thm}[equation]{Theorem}
\newtheorem{cor}[equation]{Corollary}
\newtheorem{ques}[equation]{Question}

\theoremstyle{definition}
\newtheorem{defn}[equation]{Definition}

\newtheorem{remark}[equation]{Remark}

\newtheorem*{claim}{Claim}
\begin{document}

\title[Commensurators of abelian subgroups]{Commensurators of abelian subgroups in CAT(0) groups}

\author[J. Huang]{Jingyin Huang}

\address{Department of Mathematics, The Ohio State University, 100 Math Tower, 231 W 18th Ave, Columbus, OH 43210, U.S.}

\email{huang.929@osu.edu}

\author[T. Prytu{\l}a]{Tomasz Prytu{\l}a}

\address{Max Planck Institute for Mathematics, Vivatsgasse 7, 53111 Bonn, Germany}

\email{prytula@mpim-bonn.mpg.de}



\begin{abstract}
We study the structure of the commensurator of a virtually abelian subgroup $H$ in $G$, where $G$ acts properly on a $\mathrm{CAT}(0)$ space $X$. When $X$ is a Hadamard manifold and $H$ is semisimple, we show that the commensurator of $H$ coincides with the normalizer of a finite index subgroup of $H$. When $X$ is a $\mathrm{CAT}(0)$ cube complex or a thick Euclidean building and the action of $G$ is cellular, we show that the commensurator of $H$ is an ascending union of normalizers of finite index subgroups of $H$. We explore several special cases where the results can be strengthened and we discuss a few examples showing the necessity of various assumptions. Finally, we present some applications to the constructions of classifying spaces with virtually abelian stabilizers.

\end{abstract}

\maketitle

\section{Introduction}
\label{sec:intro}
\subsection*{Background and motivation}
We say that two subgroups $H_1, H_2$ of a group $G$  are \emph{commensurable} if $H_1 \cap H_2$ has finite index in both $H_1$ and $H_2$. The \emph{commensurator} of $H$ in $G$, denoted by $\comm{G}{H}$, is a subgroup consisting of all elements $g \in G$ such that $gHg^{-1}$ and $H$ are commensurable. In this article we would like to understand $\comm{G}{H}$ when $H$ is virtually abelian and $G$ acts properly on a $\czero$ space $X$.


One motivation for studying such commensurators comes from the connection between some of their properties and the topology of classifying spaces of $G$ with respect to families of virtually abelian subgroups \cite{MR2900176}. 

Another motivation comes from $\czero$ geometry. For $\czero$ groups, the normalizers of their abelian subgroups are well-understood and they play a fundamental role in the theory of $\czero$ groups \cite{bh}. However, the commensurators of abelian subgroups are much more mysterious and they contain subtle information of the action which is not seen by the normalizers.

The commensurator $\comm{G}{H} \le G$ contains normalizers of finite index subgroups of $H$. It is therefore natural to ask how far the commensurator is from being a normalizer. In general $\comm{G}{H}$ may not be finitely generated for a $\czero$ group~$G$; such an example can be found in Wise's work on irreducible lattices acting on product of trees \cite{wise}, we refer to Proposition~\ref{prop:anti torus} for an explanation. On the other hand, the normalizer of $H$ is always finitely generated \cite{bh}. Thus we ask about finitely generated subgroups of the commensurator instead. This leads to the following, which is a generalization of L{\"u}ck's $\conc$ for cyclic subgroups \cite{luckcat0}.

\begin{defn}[$\conc$]We say that a virtually abelian subgroup $H$ of $G$ satisfies $\conc$ if every finitely generated subgroup $K \le \comm{G}{H}$ normalizes some finite index subgroup of $H$. A group $G$ satisfies $\conc$ if each of its finitely generated virtually abelian subgroup satisfies $\conc$.
\end{defn}

If $\conc$ holds for $H$, then $\comm{G}{H}$ is an ascending union of normalizers of finite index subgroups of $H$. Note that if $\comm{G}{H}$ does not have pathologies of type I (not being finitely generated) and type II (not satisfying $\conc$), then $\comm{G}{H}$ is equal to the normalizer of a finite index subgroup of $H$. 

Recently, Leary and Minasyan gave an example of a $\czero$ group which does not satisfy $\conc$ \cite{LM}, so both types of pathologies can occur for $\czero$ groups. However, there are many natural classes of $\czero$ groups where such pathologies can be eliminated due to certain geometric or combinatorial structure of these groups, which we will discuss below. We will also indicate several results for general $\czero$ groups, including an application to Bredon cohomological dimension for virtually abelian stabilizers.


\subsection*{Structure of the commensurators}
A group $G$ acts \emph{geometrically} on a metric space $X$ if it acts isometrically, properly and cocompactly. 
A subgroup $H \leq G$ is \emph{semisimple} if each element of $H$ is a semisimple isometry of $X$.  

Recall that a \emph{Hadamard manifold} is a complete, simply connected smooth manifold with non-positive sectional curvature. For groups acting geometrically on Hadamard manifolds, none of the above pathologies occurs. Actually, a slightly more general result holds.

\begin{thm}[Theorem~\ref{thm:hadamard}]\label{introthm:hadamard}
	Suppose $G$ acts properly on a Hadamard manifold $X$ by isometries. Let $H\le G$ be a semisimple finitely generated virtually abelian subgroup. Then $\comm{G}{H}$ is equal to the normalizer of a finite index subgroup of $H$. In particular, if the action $G\acts X$ is geometric then $\comm{G}{H}$ is finitely generated.
\end{thm}

This theorem fails if we relax the assumption of `Hadamard manifold' to `complete $\czero$ manifold without boundary', see Corollary~\ref{cor:LM cor}. We remark that Theorem~\ref{introthm:hadamard} gives obstructions to (virtually) embed a group into a fundamental group of a non-positively curved smooth closed manifold. 


For singular $\czero$ spaces the structure of commensurators is generally complicated, even if the space admits a piecewise Euclidean structure. However, certain types of piecewise Euclidean structures give rise to rigidity of commensurators.

\begin{thm}[Theorem~\ref{thm:cubicalpropc} and Proposition~\ref{prop:symmetricspace}]
	\label{introthm:symmetricspace}
Suppose $G$ acts properly on a $\czero$ space $X$. Let $H\le G$ be a finitely generated virtually abelian subgroup. Suppose one of the following is satisfied.
\begin{enumerate}
	\item $X$ is a finite dimensional $\czero$ cube complex and $G$ acts on $X$ by cubical automorphisms. 
	\item $X=X_1\times X_2\times\cdots \times X_n$ such that each $X_i$ is either a nonflat irreducible symmetric space of noncompact type or an irreducible thick Euclidean Tits building with cocompact affine Weyl group, and $H$ is semisimple.
\end{enumerate}
Then any finitely generated subgroup of $\comm{G}{H}$ normalizes a finite index subgroup of $H$. 
\end{thm}

Note that the conclusion is slightly weaker than in Theorem~\ref{introthm:hadamard}. In both cases of Theorem~\ref{introthm:symmetricspace}, there are examples of $\comm{G}{H}$ being not finitely generated, even if the action of $G$ on $X$ is cocompact. Moreover, one cannot remove the `finitely generated' assumption from the last sentence of Theorem~\ref{introthm:symmetricspace}, that is, the commensurator $\comm{G}{H}$ may not normalize any finite index subgroup $H'\leq H$, as shown in Proposition~\ref{prop:anti torus}. However, such assumption can be removed when the action of $G$ on a cube complex is virtually special, see Corollary~\ref{cor:centralizer}.





We now look at more general $\czero$ spaces. A finitely generated virtually abelian subgroup is \emph{highest}, if it does not have a finite index free abelian subgroup that lies in a free abelian subgroup of higher rank.

\begin{prop}[Proposition~\ref{prop:highestvab}]\label{introprop:highestvab}
Let $G$ be a group acting geometrically on a $\czero$ space and suppose $H$ is a highest virtually abelian subgroup of $G$. Then $\comm{G}{H}$ contains $H$ as a finite index subgroup. In particular $\comm{G}{H}$ is finitely generated and it normalizes a finite index subgroup of $H$.
\end{prop}

The assumption of being highest cannot be removed \cite{LM}. Also note that a highest abelian subgroup might not be `highest' in a geometric sense. More precisely, there is an example by Rattaggi and Robertson \cite{RattaggiRobertson} of a highest abelian subgroup $H$ in a $\czero$ group $G$ such that $H$ acts cocompactly on a flat $F\subset X$ with $F$ being contained in a higher dimensional flat, see Proposition~\ref{prop:highest}.

We also observe that for $\czero$ groups, the examples by Leary and Minasyan \cite{LM} are the only obstructions to $\conc$, see Proposition~\ref{prop:conc obstruction}.

\subsection*{Applications to the classifying spaces for families of subgroups} Given a group $G$ and a family of subgroups $\mathcal{F}$, \emph{the classifying space of $G$ for the family} $\mathcal{F}$, denoted by $E_{\mathcal{F}}G$, is the universal $G$--CW--complex with stabilizers in $\mathcal{F}$. Classifying spaces for families appear in Baum-Connes and Farrell-Jones isomorphism conjectures in $K$--theory and they can be used to compute Bredon cohomology of $G$ \cite{lucksurv}. Therefore it is desirable to construct simple models for $E_{\mathcal{F}}G$ and in particular to bound its dimension. The minimal dimension of $E_{\mathcal{F}}G$ is called \emph{geometric dimension} of $G$ for the family $\mathcal{F}$, and is denoted by $\mathrm{gd}_{\mathcal{F}}G$. There is an algebraic counterpart of geometric dimension called \emph{Bredon cohomological dimension} and it is denoted by $\mathrm{cd}_{\mathcal{F}}G$. These two dimensions are related by Eilenberg-Ganea-type inequality 
\begin{equation*}
\mathrm{cd}_{\mathcal{F}}G \le \mathrm{gd}_{\mathcal{F}}G \le \mathrm{max}\{3, \mathrm{cd}_{\mathcal{F}}G  \}.
\end{equation*}
The standard by now method to construct a classifying space  $E_{\mathcal{F}}G$ or to bound  $\mathrm{cd}_{\mathcal{F}}G$ is a construction due to L{\"u}ck and Weiermann~\cite{MR2900176}. One may say that the key point of that construction is the study of $\comm{G}{H}$ for subgroups $H \in \mathcal{F}$ and construction of classifying spaces for this group for certain families (simpler than~$\mathcal{F}$).

In vast majority of constructions, the following approach has been used: one first proves $\conc$ for $G$ and then approximates $\comm{G}{H}$ by normalizers of subgroups $H_i$ commensurable with $H$. Then one constructs respective classifying spaces for normalizers, as they are usually much simpler groups, and finally one reconstructs classifying space for $\comm{G}{H}$ from classifying spaces for $N_G(H_i)$ for subgroups $H_i$.

Now let $\mathcal{F}_r$ be a family of subgroups of $G$ which consists of all finitely generated virtually abelian subgroups of rank at most $r$.
Following the above procedure, the second author obtained a bound on the $\mathrm{cd}_{\mathcal{F}_r}G$ for $G$ acting properly by semisimple isometries on a proper finite dimensional $\czero$ space, assuming that $G$ satisfies $\conc$ \cite[Theorem~1.1]{cat0vab}.

However, by studying the action of $\comm{G}{H}$ more carefully, now we can remove $\conc$ from the assumptions of \cite[Theorem~1.1]{cat0vab}.

\begin{thm}[Theorem~\ref{thm:improvedcat0vab}]\label{introthm:improvedcat0vab} Let $G$ be a group acting properly by semisimple isometries on a complete proper $\czero$ space of topological dimension $n$. Then for any $0 \leqslant r \leqslant n$ we have $\mathrm{cd}_{\mathcal{F}_r}G \leq n+r+1.$
\end{thm}

Let us point out that Theorem~\ref{introthm:improvedcat0vab} gives a partial answer to a question by Lafont \cite{lafont2008construction}, concerning constructions of classifying spaces for the family of virtually abelian subgroups.


\subsection*{Comments on the proof}
Suppose $G$ acts properly on a $\czero$ space $X$ and suppose $H\le G$ is a semisimple finitely generated virtually abelian subgroup. Let $F\subset X$ be an $H$--invariant flat such that $H\acts F$ is cocompact. Then $\comm{G}{H}$ preserves the parallel set $P_F=F\times F^{\perp}$ and its product structure (Lemma~\ref{lem:split}). This gives rise to two factor actions $\comm{G}{H}\acts F$ and $\comm{G}{H}\acts F^\perp$. Theorem~\ref{introthm:symmetricspace} and Proposition~\ref{introprop:highestvab} come from analyzing the regularity of $\comm{G}{H}\acts F$; and Theorem~\ref{introthm:hadamard} and Theorem~\ref{introthm:improvedcat0vab} come from analyzing the regularity of $\comm{G}{H}\acts F^\perp$.

\subsection*{Organization of the paper}
In Section~\ref{sec:C} we give background on $\conc$. In Section~\ref{sec:honolomy} we collect several preparatory observations for later sections. Sections~\ref{sec:smoothmfds}, \ref{sec:cube}, \ref{sec:buildings}, \ref{sec:bredon} and~\ref{sec:lmgroups} are essentially independent from one another. In Sections~\ref{sec:smoothmfds},~\ref{sec:cube} and~\ref{sec:buildings} we handle the cases of Hadamard manifolds, cube complexes and Euclidean buildings respectively. Section~\ref{sec:bredon} is about applications to Bredon cohomological dimension. In Section~\ref{sec:lmgroups} we discuss the relation between $\conc$ and the examples by Leary and Minasyan. More examples of possible pathological behavior of commensurators are given in Section~\ref{sec:examples}.

\subsection*{Acknowledgments}

J.\ H.\ thanks the Max Planck Institute for Mathematics where part of the work was completed. J.\ H.\ thanks J.\ Lafont, T.\ Nguyen and T.\ T.\ Nguyen-Phan for helpful discussions. 

T.\ P.\ thanks the Fields Institute for Research in Mathematical Sciences where part of the work was completed.
T.\ P.\ was supported by EPSRC First Grant EP/N033787/1.  T.\ P.\ thanks G.\ Margulis and J.\ Schwermer for helpful discussions.

Both authors thank I.\ Leary and A.\ Minasyan for valuable discussions and comments improving the paper.


\section{Background on $\conc$}
\label{sec:C}
We refer to the \hyperref[sec:intro]{Introduction} for definitions of a commensurator and $\conc$. Throughout, we will be using the following simple observation.
\begin{lem}
	\label{lem:comm invariant}
Let $G$ be a group and let $H_1,H_2$ be two finitely generated virtually abelian subgroups which are commensurable. Then $H_1$ satisfies $\conc$ if and only if $H_2$ satisfies $\conc$.
\end{lem}

If $\conc$ holds for all finitely generated virtually abelian subgroups of rank equal to (respectively, at most) $r$ then we denote it by $\mathrm{(C)}_{r}$ (respectively, $\mathrm{(C)}_{\leq r})$. Condition $\mathrm{(C)}_{1}$ essentially boils down to showing that for any infinite order element $h \in G$, whenever \[gh^kg^{-1} =h^l\] for some $g \in G$ and $k,l\neq 0$, then $k= \pm l$. This can be easily shown if one can assign to every such $h$ a `norm' which satisfies $\|h^n\|=|n| \cdot \|h\|$ and is invariant under conjugation. In several classes of non-positively curved groups such norm is given by (different variants of) \emph{translation length}, and so condition $\mathrm{(C)}_{1}$ is satisfied by $\czero$ groups, $\delta$--hyperbolic groups, systolic groups and biautomatic groups. A simple example of a group which does not satisfy $\mathrm{(C)}_{1}$ is the Baumslag-Solitar group  $BS(1,n)$, for $n>1$. \medskip

When passing to higher rank abelian subgroups, the translation length alone is insufficient. In this case elements of $\comm{G}{H}$ may also `rotate' various subgroups of $H$, and in fact this condition is not always satisfied by non-positively curved groups. To the best of our knowledge the only general method of showing $\conc$ is \cite[Corollary~9]{ckrw}, which states that $H$ satisfies $\conc$ if $H$ is weakly separable in $G$. Let us point out that this result does not require $H$ to be virtually abelian. On the other hand, combining this result with the fact that virtually abelian subgroups of linear groups are separable implies $\conc$ for any linear group (see \cite{cat0vab} for a short account of the proof). This applies to, among others, Coxeter groups, graph products of finite groups, or fundamental groups of special cube complexes. One easily finds examples of groups $G$ where $H$ satisfies $\conc$ but $H$ is not weakly separable in $G$.

\section{General $\czero$ case}
\label{sec:honolomy}
\subsection{Finitely generated subgroups of commensurators}
\label{subsec:honolomy}
Let $X$ be a $\czero$ space and let $G$ be a group acting properly on $X$.  
A subgroup $H$ of $G$ is \emph{semisimple} if each element of $H$ acts as a semisimple isometry.

Let $H\le G$ be a semisimple free abelian group of rank $n$ and let $\minset{H}$ denote the minimal set of $H$. 
It is a standard fact that $\minset{H}$ splits as $\mathbb E^n\times Y$ where $Y$ is a $\czero$ space, moreover, $H$ acts freely and cocompactly by translations on the $\mathbb E^n$--factor and acts as the identity on the $Y$--factor. Let $F$ be a flat of form $\mathbb E^n\times\{y\}$ for $y\in Y$. Then $H$ stabilizes $F$. Let $P_F=F\times F^\perp$ be the parallel set of~$F$.

We will need the following theorem.

\begin{thm}\label{thm:combinenormalizerminset}
Let $N_G(H)$ denote the normalizer of $H$ in $G$ and let $Z_G(H)$ denote the centralizer of $H$ in $G$. Then:


\begin{enumerate}[label=(\roman*)]

\item \label{thm:centralizer}  The index $[N_G(H):Z_G(H)]$ is finite. In particular, if  a subgroup $\Gamma \le G$ normalizes $H$ then $\Gamma$ has a finite index subgroup which centralizes $H$.

\item \label{thm:ruane}
The normalizer $N_G(H)$ preserves $\minset{H} \cong \mathbb E^n\times Y$ and its product structure. 
If the action of $G$ on $X$ is in addition cocompact then $Z_G(H)$ acts geometrically on $\minset{H}$ (and thus by \ref{thm:centralizer} the same holds for $N_G(H)$).

\end{enumerate}

\end{thm}

\begin{proof}

\ref{thm:centralizer} and the first assertion of \ref{thm:ruane} is proven in \cite[Theorem II.7.1]{bh}. In the original statement, the $G$--action is required to be faithful, but this condition is not necessary. The second assertion of \ref{thm:ruane} can be proven in a similar way to \cite[Theorem 3.2]{ruane2001dynamics}.
\end{proof}

We now begin studying the action of $\comm{G}{H}$ on the parallel set $P_F$.

\begin{lem}
	\label{lem:split}
The set $P_F$ is invariant under $\comm{G}{H}$. Moreover, for each element $\alpha\in \comm{G}{H}$, the action of $\alpha$ on $P_F$ splits as a product of an isometry of $F$ and an isometry of $F^\perp$.
\end{lem}

\begin{proof}
Let $\alpha \in \comm{G}{H}$. To prove both assertions of the lemma, it is enough to show that for any point $y \in F^\perp$, flats $F\times\{y\} $ and $\alpha(F\times\{y\}) $ are parallel. Indeed, in this case flat $\alpha(F\times\{y\})$ is clearly contained in $P_F$, and thus it is of form $F \times \{y'\}$ for some $y' \in F^\perp$. Then \cite[Proposition 5.3(3)]{bh} implies that $\alpha$ splits as a product.

 Consider an $H$--invariant flat $F \times \{y_0\} \subset P_F$.
The flats $F \times \{y\}$ and $F \times \{y_0\}$ are parallel, and thus so are flats $\alpha(F \times \{y\})$ and $\alpha(F \times \{y_0\})$. Since parallelism is an equivalence relation, to show that $F\times\{y\} $ and $\alpha(F\times\{y\}) $ are parallel it suffices to show that $F\times\{y_0\} $ and $\alpha(F\times\{y_0\}) $ are parallel.
Since $F\times\{y_0\} $ is $H$--invariant we get that $\alpha(F\times\{y_0\}) $ is $\alpha H \inv{\alpha}$--invariant. The intersection $H'=H \cap \alpha H \inv{\alpha}$ is semisimple and is of finite index in both $H$ and $\alpha H \inv{\alpha}$. Note that both $\alpha(F\times\{y_0\})$ and 
 $\alpha(F\times\{y_0\}) $ are $H'$--invariant. Therefore they are parallel.\end{proof}

\begin{defn}
	\label{def:honolomy}
By Lemma~\ref{lem:split}, we have a well-defined homomorphism \[\phi \colon \comm{G}{H} \to \Isom(F)\] by considering the action on the $F$--factor of $P_F$. Note that each element of $\Isom(F)$ acts on the Tits boundary $\partial_T F$ of $F$, which induces a homomorphism $\Isom(F)\to O(n,\mathbb R)$ where $n=\dim F$. Let $\Phi$ be the composition $\comm{G}{H} \to \Isom(F)\to O(n,\mathbb R)$.
\end{defn}

The following observation is also of independent interest.
\begin{lem}
	\label{lem:comm}
	We view $H$ as a subgroup of $\Isom(F)$. Then the image of $\phi$ (see Definition~\ref{def:honolomy}) is contained in the commensurator of $H$ in $\Isom(F)$.
\end{lem}

\begin{proof}
	Let $g\in \comm{G}{H}$. Then $\phi(g)$ is a composition $F\overset{g}{\to} gF\overset{p}{\to} F$ where $p$ is the parallelism map. Since $g\in\comm{G}{H}$, there exists a finite index subgroup $L\le gHg^{-1}$ such that $L\le H$. Let $L'\le H$ be the finite index subgroup such that $gL'g^{-1}=L$ and let $\alpha:L'\to L$ be the isomorphism induced by conjugation by $g$. Then the map $F\overset{g}{\to} gF$ is equivariant with respect to $L'\acts F$, $L\acts gF$ and $\alpha:L'\to L$. The map $p:gF\to F$ is equivariant with respect to the action of $L$ on both $gF$ and $F$. Thus $\phi(g)$ is equivariant with respect to $L'\acts F$, $L\acts F$ and $\alpha:L'\to L$. Thus $\phi(g)L'(\phi(g))^{-1}=L$ when viewed as subgroups of $\Isom(F)$. Since $L'$ and $L$ are finite index subgroups of $H$, the lemma follows.
\end{proof}

\begin{prop}
	\label{prop:normalize}
Suppose $G$ acts properly on a $\czero$ space $X$. Let $H\le G$ be a semisimple free abelian subgroup of finite rank. Let $K\le \comm{G}{H}$ be a finitely generated subgroup. Then $K$ normalizes a finite index subgroup of $H$ if and only if $\Phi(K)$ is finite.
\end{prop}

\begin{proof}

First suppose that $K$ normalizes a finite index subgroup $H'$ of $H$. Then by Theorem~\ref{thm:combinenormalizerminset}.\ref{thm:centralizer} there exists a finite index subgroup $K'\le K$ which centralizes $H'$. We will show that $\Phi(K')= \{e\}$. This clearly implies that $\Phi(K)$ is finite.

Let $g\in K'$ be arbitrary. Since $g$ centralizes $H'$, proceeding as in the proof of Lemma~\ref{lem:comm}, we get that $\phi(g)$ is an $H'$--equivariant isometry of $F$. Since $\phi(g)$ commutes with linearly independent translations whose axes span $F$ (consider generators of $H'$), it is not hard to see that it has to be a translation itself. Thus its restriction to the boundary $\partial_TF$ is trivial. Since $g$ was arbitrary, we get that $\Phi(K') = \{e\}$.\medskip

Now suppose the image $\Phi(K)$ is finite. Thus there is a finite index subgroup $K' \le K$ for which $\Phi(K')=\{e\}$. This means that for any $g \in K'$ the isometry $\phi(g) \in \Isom(F)$ is a translation. Thus for any $g\in K'$ and $h\in H$, we have 
\begin{equation}\label{eq:translationscommute}\phi(h)=\phi(g)\phi(h)(\phi(g))^{-1}=\phi(ghg^{-1}),
\end{equation} where the first equality follows from the fact that by definition $\phi(h)$ is also a translation, hence it commutes with $\phi(g)$.

Since $K' \le \comm{G}{H}$, for any $g \in K'$ the intersection $H_g = H \cap gH\inv{g}$ is a finite index subgroup of $H$. Take an element $h\in H_g$. Thus we have $h\in H$ and also $h=gh'\inv{g}$ for some $h' \in H$. By \eqref{eq:translationscommute} we get $\phi(h')=\phi(gh'\inv{g})=\phi(h)$. Since $h,h'\in H$ and $\phi|_H$ is an embedding we obtain $h'=h$. Therefore $g$ centralizes $H_g$.

Let $g_1, \ldots, g_n$ be a set of generators of $K'$.
Every $g_i$ centralizes $H_{g_i}$ and thus $K$ centralizes the intersection $H' =\bigcap_{g_i}H_{g_i}$, which is a finite index subgroup of $H$. Now the following elementary lemma completes the proof.
\end{proof}

\begin{lem}
	\label{lem:ftindex}
	Let $H\le K' \le K \le \comm{G}{H}$ and suppose that $H$ is central in $K'$ and that index $[K \colon K']$ is finite. Then there exists a finite index subgroup $H' \le H$ which is normal in $K$. 
\end{lem}

\begin{proof}

Let $e=s_1, s_2, \ldots, s_n \in K$  be representatives of left cosets of $K'$ in $K$.
Define \[H'= \bigcap_{i}s_iH\inv{s_i}.\]
Since $K \le \comm{G}{H}$, clearly $H'$ has finite index in $H$. It is straightforward to check that $H'$ is normal in $K$.
\end{proof}

\begin{remark}\label{rem:bounddiam}
	If $F^{\perp}$ has bounded diameter, then $\comm{G}{H}$ acts on  $F\times F^\perp$ cocompactly, as $H\le \comm{G}{H}$ and $H$ acts cocompactly on $F\times \{c_0\}$ and thus on $F\times F^\perp$. It follows that $\comm{G}{H}$ contains $H$ as a finite index subgroup. Then $\comm{G}{H}$ is clearly finitely generated, and it normalizes a finite index subgroup of $H$. 
\end{remark}


\subsection{The highest virtual abelian subgroups}
A virtually abelian subgroup $H \le G$ is called \emph{highest} if it is not virtually contained in a virtually abelian subgroup $H' \leq G$ of higher rank. In this section we show that the commensurator of a highest virtually abelian subgroup always behaves nicely. We need the following well-known lemma. We give a proof for the sake of completeness.

\begin{lem}\label{lem:highest}Let $G$ be a group acting geometrically on a $\czero$ space $X$ and suppose $H$ is a highest abelian subgroup of $G$. Then the normalizer $N_G(H)$ contains $H$ as a finite index subgroup.
\end{lem}

\begin{proof}Since $G$ acts geometrically on $X$, by Theorem~\ref{thm:combinenormalizerminset}.\ref{thm:ruane} the normalizer $N_G(H)$ acts geometrically on the minimal set  $\minset{H} \cong \mathbb E^n\times Y$. Since the action preserves the splitting and $H$ acts on $Y$ trivially, it follows that $N_G(H)/H$ acts geometrically on the $\czero$ space $Y$ (see \cite[Section~II.7]{bh}). Now suppose $[N_G(H) \colon H]$ is infinite.  In this case $N_G(H)/H$ is an infinite $\czero$ group and therefore it contains an element of infinite order $g$ \cite[Theorem~11]{Swenson}. Let $\tilde{g} \in N_G(H)$ be any preimage of $g$. Since by Theorem~\ref{thm:combinenormalizerminset}.\ref{thm:centralizer} the index $[N_G(H) \colon Z_G(H) ]$ is finite we get that some power $\tilde{g}^n$ commutes with $H$, thus contradicting the fact that $H$ is highest.
\end{proof}

\begin{prop}\label{prop:highestvab}
Let $G$ be a group acting geometrically on a $\czero$ space and suppose $H$ is a highest virtually abelian subgroup of $G$. Then $\comm{G}{H}$ contains $H$ as a finite index subgroup. In particular $\comm{G}{H}$ is finitely generated, and it normalizes a finite index subgroup of $H$.
\end{prop}

\begin{proof}Without loss of generality we can assume that $H$ is abelian. 	Let $\Phi \colon \comm{G}{H} \to O(n, \mathbb{R})$ be the map given in Definition~\ref{def:honolomy}. By Proposition~\ref{prop:normalize} any finitely generated subgroup $L \le \mathrm{ker}(\Phi)$ normalizes some finite index subgroup of $H$. Therefore $\mathrm{ker}(\Phi)$ can be written as an ascending union $\bigcup_i L_i$ such that $L_i$ normalizes some finite index subgroup $H_i \le H$. We can assume that every $L_i$ contains $H$ (replace $L$ with $\langle L, H\rangle)$. 
 Since any $H_i$ is highest, it follows from Lemma~\ref{lem:highest} that $N_G(H_i)$ is a finite extension of $H_i$, and thus $L_i$ is a virtually abelian group of rank $\rk{H}$ since $H_i \le H \le L_i \le N_G(H_i)$. We obtain that $\mathrm{ker}(\Phi)$ is an ascending union of finitely generated virtually abelian groups of rank $\rk{H}$. By the Ascending Chain Theorem \cite[Theorem~II.7.5]{bh} this ascending union stabilizes after finitely many terms. Thus $\mathrm{ker}(\Phi)$ is a finitely generated virtually abelian group or rank $\rk{H}$.
 
Since the index $[\mathrm{ker}(\Phi) \colon H]$ is finite, we can find a finite index characteristic subgroup $H'$ of $ \mathrm{ker}(\Phi)$ with $H'\le H$. Now since $H'$ is characteristic in $\mathrm{ker}(\Phi)$ and $\mathrm{ker}(\Phi)$ is normal in $\comm{G}{H}$, it follows that $H'$ is normal in $\comm{G}{H}$. By Lemma~\ref{lem:highest} the index $[\comm{G}{H} \colon H']$ is finite and thus $[\comm{G}{H} \colon H]$ is finite as well.  
\end{proof}

We remark that Proposition~\ref{prop:highestvab} is not a consequence of Remark~\ref{rem:bounddiam}, see Proposition~\ref{prop:highest}.

\begin{cor}\label{cor:2dimconc} Let $G$ be a group that either
\begin{enumerate}
 \item acts properly by semisimple isometries on a $2$--dimensional $\czero$ space, or
\item acts  geometrically on a $\czero$ space and contains no subgroup isomorphic to $\mathbb{Z}^n$ for $n>2$.
\end{enumerate}
Then $\conc$ holds for $G$.
\end{cor}

\begin{proof}Since $G$ does not contain free abelian subgroups of rank higher than $2$, $\conc$ is equivalent to $\conch{\le 2}$. Since $G$ acts properly by semisimple isometries on a $\czero$ space, $\conch{\le 1}$ holds for $G$ (see Section~\ref{sec:C}). Let $H$ be a rank $2$ free abelian subgroup of $G$. In the first case $\conc$ for $H$ is satisfied by Remark~\ref{rem:bounddiam}. In the second case one observes that $H$ is the highest abelian subgroup and thus $\conc$ for $H$ follows from Proposition~\ref{prop:highestvab}.
\end{proof}

\subsection{Core of \texorpdfstring{$F^\perp$}{$F$-`perp'}}
\label{subsec:core}
Let $G,H$ and $P_F=F \times F^\perp$ be as in Section~\ref{subsec:honolomy}. In this section we look at the action $\comm{G}{H}$ on $P_F$ more closely. By Lemma~\ref{lem:split}, there is a factor action $\rho:\comm{G}{H}\acts F^\perp$. For each $x\in F^\perp$, let $\Stab_\rho(x)$ be the stabilizer of $x$ with respect to the action $\rho$.

\begin{defn}
We define the \emph{core} of $F^\perp$, denoted $\core$, to be the subset of $F^\perp$ made of points whose stabilizer (with respect to $\rho$) is commensurable to $H$.
\end{defn}

\begin{lem}
\label{lem:core}

The core $\core  \subset F^{\perp}$ is non-empty, convex and $\comm{G}{H}$--invariant.
\end{lem}
\begin{proof}
Clearly $\core$ is non-empty. Note that $\Stab_\rho(gx)=g\Stab_\rho(x)g^{-1}$ for any $g\in \comm{G}{H}$, thus $\core$ is $\comm{G}{H}$--invariant. To see that $\core$ is convex, choose $c_1,c_2\in C$ and let $c_0\in \core$ be a point in the geodesic segment $\overline{c_1c_2}$. For $0\le i\le 2$, let $H_i=\Stab_\rho(c_i)$, and let $H'_0=H_1\cap H_2$. Clearly $H'_0\le H_0$ and $H'_0$ is commensurable to $H$. Since $H'_0$ acts cocompactly on $F\times\{c_0\}$ and $H_0$ acts properly on $F\times\{c_0\}$, $H'_0$ is of finite index in $H_0$. Thus $H_0$ is commensurable to $H$ and so $c_0\in \core$.
\end{proof}

Note that in general $\core$ is not complete or closed in $F^{\perp}$.

\begin{lem}\label{lem:discreteorb}Assume that $\core$ is a proper metric space. Then the action $\comm{G}{H}\acts \core$ has discrete orbits.
\end{lem}

\begin{proof}

Suppose the contrary, that there exists $y_0 \in \core$ such that for any $\epsilon>0$ there are infinitely many distinct elements $(g_i)_{i\in \mathbb{N}}$ such that $d(g_iy_0, y_0) < \epsilon$. 
 Embed $\core \hookrightarrow F \times \core \subset P_F$ by $y \mapsto (x_0, y)$ for some chosen $x_0 \in F$ and consider the action of $\comm{G}{H}$ on $F \times \core$. We will suppress from writing $\phi$ and $\rho$ and simply write $g(x,y)=(gx,gy)$ for this action.
 We need the following claim.\medskip
 
\textbf{Claim.} There exists a constant $R>0$ such that for any $g_i \in \comm{G}{H}$, given any two points $(a, g_iy_0) ,(b, g_iy_0) \in F \times \{g_iy_0\}$ there exists $h_i \in \Stab( F \times \{g_iy_0\})$ such that \[d(h_i(a, g_iy_0),(b, g_iy_0))\leq R.\]
 
 \medskip
 
 To see the claim, first observe that since $y_0 \in \core$, the stabilizer $\Stab( F \times \{y_0\})$ acts cocompactly on $F \times \{y_0\}$ and therefore such constant $R$ exists for $ F \times \{y_0\}$. Since for any $g_i$ we have $\Stab( F \times \{g_iy_0\})=g_iStab( F \times \{y_0\})g_i^{-1}$ and the $\Stab( F \times \{g_iy_0\})$--action on $F \times \{g_iy_0\}$ is conjugate to the $\Stab( F \times \{y_0\})$ on $F \times \{y_0\}$ it follows that constant $R$ works for $F \times \{g_iy_0\}$ as well.\medskip
 
 We proceed with the proof of the lemma. By the assumption we have infinitely many distinct points $(g_ix_0, g_iy_0) \in F \times B(y_0, \epsilon)$.
 By the claim there exist elements $h_i \in  \Stab( F \times \{g_iy_0\})$ such that 
 \[d(h_i(g_ix_0, g_iy_0) , (x_0, g_iy_0)) \le R.\] Therefore infinitely many points $(h_i(g_ix_0, g_iy_0))_{i \in \mathbb{N}}$ are contained in the compact subset $B(x_0,R) \times B(y_0, \epsilon) \subset F\times \core$. These points are distinct because their second coordinates are distinct. Consequently, all the elements $(h_ig_i)_{i \in \mathbb{N}}$ are distinct, which contradicts the properness of the action $\comm{G}{H}$ on $ F \times C$.\end{proof}

\section{The smooth manifolds case}\label{sec:smoothmfds}
In this section we show that commensurators of abelian subgroups are well-behaved for groups acting on Hadamard manifolds.

Let $M$ be a Riemannian manifold without boundary and let $C\subset M$ be a \emph{totally convex} subset, i.e.,\ for any pair of points $x,y\in C$ and any Riemannian geodesic $\omega$ connecting $x$ and $y$, we have $\omega\subset C$. 

Let $k$ be the largest integer such that the collection $\{N_\alpha\}$ of smoothly embedded $k$--manifolds of $M$ which are contained in $C$ is non-empty. Let $N=\cup_\alpha N_\alpha$. The following result is well-known (see e.g.\ \cite[pp. 139 - 141]{cheeger2008comparison}).

\begin{lem}
\label{lem:convex subset} The subset $N$ is a totally geodesic, connected, smoothly embedded submanifold of $M$ such that $N\subset C \subset \bar N$, where $\bar N$ is the closure of $N$ in $M$.
\end{lem}

If $C$ is a point, then $N=C$ is also a point.

\begin{thm}\label{thm:hadamard}
Suppose $G$ acts properly on a Hadamard manifold $X$ by isometries. Let $H\le G$ be a semisimple finitely generated virtually abelian subgroup. Then $\comm{G}{H}$ is equal to the normalizer of a finite index subgroup of $H$. In particular, if the action $G\acts X$ is geometric then $\comm{G}{H}$ is finitely generated.
\end{thm}

Recall that isometries of Riemannian manifolds as metric spaces are actually diffeomorphisms and preserve the Riemannian tensor.

\begin{proof}Without loss of generality we can assume that $H$ is free abelian. Let $F$ be a flat in the minimal set of $H$ where $H$ acts cocompactly. Let $P_F = F\times F^\perp$ be the parallel set of $F$. As in Section~\ref{subsec:core}, let $\rho:\comm{G}{H}\acts F^{\perp}$ be the factor action and let $\core \subset F^\perp$ be the core. 
	
Choose a basepoint $y\in F$. Then $F^\perp$ (respectively $\core$) can be realized as a convex subset $\{y\}\times F^\perp$ (respectively $\{y\}\times \core$) of $X$. Let $p:P_F\to \{y\}\times F^\perp$ be the projection map. Let $N=\cup_\alpha N_\alpha \subset \core$ be as in Lemma~\ref{lem:convex subset}. Since $N$ is totally geodesic, $N$ is convex in $\core$. For any $g\in \comm{G}{H}$, the composition $N_\alpha\to (p\circ g) (N_\alpha)$ is a diffeomorphism. Thus $\rho:\comm{G}{H}\acts F^{\perp}$ leaves $N$ invariant and acts on $N$ by Riemannian isometries.
	
Choose $x\in N$. Since $N \subset \core$, the stabilizer $\Stab_\rho(x)$ is commensurable to $H$. Let $T_x N$ be the tangent space of $N$ at $x$. Let $H'$ be the kernel of the natural homomorphism from $\Stab_\rho(x)$ to orthogonal group of $T_x N$. 
\begin{claim}
The subgroup $H'$ is of finite index in $\Stab_\rho(x)$.\end{claim}
To see the claim, take an orthogonal frame ${e_1,e_2,\ldots, e_n}$ in $T_x N$ (it is possible that $n=0$), and find ${x_1,x_2,\ldots,x_n}$ in $N$ such that the geodesic segment from $x$ to $x_i$ has tangent vector $e_i$ at $x$ (such $x_i$ exists since $N$ is a manifold without boundary). Since $x_i \in N \subset \core$, the subgroup $\Stab_\rho(x_i)$ is commensurable to $\Stab_\rho(x)$. Thus a finite index subgroup $H_i\le\Stab_\rho(x)$ stabilizes $x_i$, hence also stabilizes $e_i$. It follows that $\cap_{i=1}^nH_i$ stabilizes all of $\{e_1,e_2,\ldots,e_n\}$ and thus the claim follows.\medskip

Let $h:\comm{G}{H}\to\Isom(N)$ be the homomorphism induced by $\rho$. Since $N$ is a smooth manifold, $H'$ acts trivially on $N$. Thus $H'\le \ker(h)$. On the other hand, $\ker(h) \le \Stab_\rho(x)$. Since $H'$ is of finite index in $\Stab_\rho(x)$ and $\Stab_\rho(x)$ is commensurable to $H$, we obtain that $\ker(h)$ is commensurable to $H$. Let $H''$ be a finite index characteristic subgroup of $\ker(h)$ such that $H''\le \ker(h)\cap H$. Then $H''$ is of finite index in $H$ and $H''$ is normalized by $\comm{G}{H}$. Since clearly we have $N_G(H'') \le \comm{G}{H}$, we conclude that $\comm{G}{H} = N_G(H'')$.

If the action $G\acts X$ is geometric then by Theorem~\ref{thm:combinenormalizerminset}.\ref{thm:ruane} the normalizer $N_G(H'')$ acts geometrically on $\minset{H''}$ and thus it is finitely generated.
\end{proof}

\section{The cube complex case}
\label{sec:cube}
\subsection{General actions}
We refer to the excellent notes by Sageev \cite{sageevnotes} for background on $\czero$ cube complexes and hyperplanes. Let $X$ be a finite dimensional $\czero$ cube complex and let $F\subset X$ be a flat.

A hyperplane $h$ \emph{crosses} $F$ if $F$ is not contained in a halfspace bounded by $h$. Note that if $h$ crosses $F$, then $h\cap F$ is a codimension 1 flat in $F$. Given $F$, let $\h(F)$ be the collection of hyperplanes that cross $F$. 

\begin{lem}
	\label{lem:intersection1}
Let $F_1$ and $F_2$	be two parallel flats. Then $\h(F_1)=\h(F_2)$. Moreover, for any $h\in \h(F_1)$, $h\cap F_1$ and $h\cap F_2$ are parallel.
\end{lem}

\begin{proof}
First we claim that if $h$ crosses $F_1$ then for any $N>0$, there exist $x_1,x_2 \in~F_1$ such that they are on different sides of $h$, and we have $d(x_1,h)>N$ and $d(x_2,h)>N$. Note that the first assertion of the lemma follows readily from this claim. To see the claim, first take $y_0\in F_1\cap h$ and $y_1,y_2\in F_1$ on different sides of $h$. Let $r_1:[0,\infty)\to F_1$ be a ray emanating from $y_0$ and passing through $y_1$. Then the convexity of the function $t\to d(r_1(t),h)$ implies that $\lim_{t\to\infty} d(r_1(t),h)=\infty$ and $r_1\cap h=\{y_0\}$. Thus we can define $x_1$ to be $r_1(t)$ for a sufficiently large $t$. Similarly we can find $x_2$.

Now we prove the `moreover' statement. Let $E$ be the convex hull of $F_1$ and $F_2$. By \cite[Chapter II.2.12]{bh}, $E$ is isometric to $F_1\times[0,a]$ where $a=d(F_1,F_2)$. Since $h$ has a product neighborhood isometric to $h\times [0,1]$, $E\cap h$ is a convex codimension 1 surface of $E$. Thus $E\cap h$ is isometric to $(h\cap F_1)\times[0,a']$ for some $a'\ge a$. Hence $h\cap F_1$ and $h\cap F_2$ are parallel.
\end{proof}

Let $G$ be a group acting on $X$ properly by cubical automorphisms. Let $H\le G$ be a free abelian subgroup and suppose $H$ acts on a flat $F\subset X$ cocompactly. 
Since for $k\in \comm{G}{H}$, flats $F$ and $k F$ are parallel, we get that $k\h(F)=\h(kF)=\h(F)$ by Lemma~\ref{lem:intersection1}. This shows that $\h(F)$ is $\comm{G}{H}$--invariant. 

For $h_1,h_2\in \h(F)$, we define $h_1\sim h_2$ if $h_1\cap F$ and $h_2\cap F$ are parallel. It is clear that $\sim$ is an equivalence relation. Since each element in $\h(F)$ intersects $F$ in a codimension 1 flat, any pair of non-equivalent hyperplanes in $\h(F)$ have non-empty intersection. Since $X$ is finite dimensional, the collection $\h(F)$ has finitely many equivalence classes, which we denote by $\{\h_i(F)\}_{i=1}^n$.

Choose a basepoint $o\in F$. For each $i$, let $\vec{v}_i$ be a non-zero vector based at $o$ such that it is orthogonal to $h\cap F$ for some $h\in \h_i(F)$.

\begin{lem}
	\label{lem:span}
The flat $F$ is spanned by $\{\vec{v}_i\}_{i=1}^n$.
\end{lem}

\begin{proof}
Suppose the contrary is true. Then there is a line $\ell\subset F$ which is orthogonal to each $\vec{v}_i$. Let $h_0$ be a hyperplane crossing $\ell$. Then $h_0\in \h(F)$. It follows from the choice of $\ell$ that $h\cap F$ contains a line parallel to $\ell$ for each $h\in\h(F)$. Thus $h_0\neq h$ and $h_0\cap h\neq\emptyset$ for each $h\in\h(F)$, which yields a contraction.
\end{proof}

\begin{lem}
	\label{lem:finite index1}
There is a finite index subgroup $L$ of $\comm{G}{H}$ such that $L(\h_i(F))=\h_i(F)$ for each $i$.
\end{lem}

\begin{proof}
An \emph{orthogonal partition} of $\h(F)$ is a partition $\h(F)=\sqcup_{i=1}^m W_i$ such that for any $i\neq j$, each element in $W_i$ crosses every element in $W_j$. Note that every two orthogonal partitions of $\h(F)$ have a common refinement which is an orthogonal partition. Thus $\h(F)$ has a canonical finest orthogonal partition $\h(F)=\sqcup_{i=1}^{l} W'_i$ (since $X$ is finite dimensional, $l<\infty$), and $\comm{G}{H}$ permutes the factors of this partition. Thus $\comm{G}{H}$ has a finite index subgroup $L$ such that $L(W'_i)=W'_i$ for $1\le i\le l$. Since $\h(F)=\sqcup_{i=1}^n \h_i(F)$ is also an orthogonal partition, the lemma follows.
\end{proof}

\begin{cor}
	\label{cor:parallel}
Let $L \le \comm{G}{H}$ be as in Lemma~\ref{lem:finite index1}. Then for each $k\in L$ and $h\in \h(F)$, $k(h\cap F)$ and $h\cap F$ are parallel.
\end{cor}

\begin{proof}
By Lemma~\ref{lem:finite index1}, $kh$ and $h$ are in the same equivalence class. Thus $F\cap h$ and $F\cap kh$ are parallel. By Lemma~\ref{lem:intersection1}, flats $F\cap kh$ and $kF\cap kh$ are parallel (as $F$ and $kF$ are parallel). Thus the corollary follows.
\end{proof}

\begin{thm}\label{thm:cubicalpropc}
Suppose $G$ acts properly on a finite dimensional $\czero$ cube complex $X$ by cubical automorphisms. Let $H\le G$ be a finitely generated virtually abelian subgroup and let $K\le \comm{G}{H}$ be a finitely generated subgroup. Then $K$ normalizes a finite index subgroup of $H$. 
\end{thm}

\begin{proof}
We can assume $H$ is free abelian by Lemma~\ref{lem:comm invariant}. Let $\phi$ and $\Phi$ be the maps in Definition~\ref{def:honolomy}. By Proposition~\ref{prop:normalize} it suffices to show that $\Phi(K)$ is finite. Define $K' = K \cap L$ where $L$ is as in Lemma~\ref{lem:finite index1}. Note that $[K : K']$ is finite since $[\comm{G}{H} : L]$ is finite.  Then Corollary~\ref{cor:parallel} implies that for any $k \in K'$ and for any $h\in \h(F)$, flats $h\cap F$ and $\phi(k)(h\cap F)$ are parallel. Let $\{\vec{v}_i\}_{i=1}^n$ be as in Lemma~\ref{lem:span}, and let $\{s_i\}_{i=1}^n$ denote the corresponding points on Tits boundary $\partial F$. Then each element in $\Phi(K')$ maps $s_i$ to $s_i$ or $-s_i$. Since $\{\vec{v}_i\}_{i=1}^n$ spans $F$, we conclude that $\Phi(K')$ is finite and hence that $\Phi(K)$ is finite.
\end{proof}

\begin{remark}
We cannot relax the assumption in Theorem~\ref{thm:cubicalpropc} that $G$ acts by cubical automorphisms to $G$ acts by isometries (though for many cube complexes these two conditions are equivalent). This is because LM groups in Definition~\ref{def:LM} clearly act on a $\czero$ cube complex by isometries, but the action does not respect the cubical structure.  	
\end{remark}

\begin{remark}
It follows from \cite{niblo1997groups} that if a group $G$ acts on Davis complex for a Coxeter group properly by cellular isometries, then $G$ satisfies the assumptions of Theorem~\ref{thm:cubicalpropc} and hence $G$ satisfies $\conc$. More generally, we speculate that by the same proof, $\conc$ should hold for groups acting properly by cellular isometries on $\czero$ piecewise Euclidean polyhedral complexes whose cells are isometric to Coxeter cells.
\end{remark}

\subsection{Virtually special actions}
In this section, we comment on an important class of actions which are \emph{virtually special}. The main point is that the pathological behavior in Proposition~\ref{prop:anti torus} cannot happen when the action is virtually special.
\begin{defn}
The action of $G$ on a $\czero$ cube complex $X$ is \emph{virtually special} if there exists a torsion free finite index subgroup $G'\le G$ such that $X/G'$ is a (not necessarily compact) special cube complex in the sense of \cite{haglund2008special}.
\end{defn}

Recall that a group $G$ has the \emph{unique root property} if for any positive integer $n$ and arbitrary elements $x,y\in G$ the equality $x^n=y^n$ implies $x=y$ in $G$. We will need the following elementary property of groups with the unique root property.

\begin{lem}
	\label{lem:commute}
	Suppose $G$ has the unique root property and let $g,h\in G$. If $g^m$ and $h^n$ commute for some non-zero integers $m$ and $n$, then $g$ and $h$ commute.
\end{lem}

If $X/G'$ is a special cube complex, then $G'$ is a subgroup of a (possibly infinitely generated) right-angled Artin group \cite[Theorem 4.2]{haglund2008special}. Since any finitely generated right-angled Artin group is biorderable \cite{duchamp1992simple}, the group $G'$ is a union of biorderable groups. As biorderable groups have the unique root property \cite[Lemma 6.3]{MR2914863}, we get that $G'$ has the unique root property.


\begin{cor}
	\label{cor:centralizer}
Suppose $G$ acts properly on a finite dimensional $\czero$ cube complex $X$ by cubical automorphisms such that the action is virtually special, or more generally $G$ has a finite index subgroup $G'$ which has the unique root property. Let $H\le G$ be a finitely generated virtually abelian subgroup. Then $\comm{G}{H}$ is equal to the normalizer of a finite index subgroup of $H$. In particular, if the action $G\acts X$ is geometric then $\comm{G}{H}$ is finitely generated.
\end{cor}

\begin{proof}
By replacing $H$ with $H \cap G'$ if necessary, we can assume that $H \le G'$. Note that in this case we have  $\comm{G'}{H}=\comm{G}{H}\cap G'$. It follows from Theorem~\ref{thm:cubicalpropc} and Theorem~\ref{thm:combinenormalizerminset}.\ref{thm:centralizer} that each finitely generated subgroup $K\le \comm{G'}{H}$ has a finite index subgroup $K'$ such that $K'$ centralizes a finite index subgroup $H'$ of $H$. Thus for any $k\in K$ and $h\in H$, there exist non-zero integers $n,m$ such that $k^n$ and $h^m$ commute. By applying Lemma~\ref{lem:commute} we get that $k$ and $h$ commute. Thus $K$ centralizes $H$. Since $\comm{G'}{H}$ is a union of finitely generated subgroups, $\comm{G'}{H}$ centralizes $H$. Since $\comm{G'}{H}$ has finite index in $\comm{G}{H}$, by Lemma~\ref{lem:ftindex} we get that $\comm{G}{H}$ normalizes a finite index subgroup $H''$ of $H$, and thus we conclude that $\comm{G}{H}=N_G(H'')$. If the action $G\acts X$ is geometric then $\comm{G}{H}$ is finitely generated by Theorem~\ref{thm:combinenormalizerminset}.\ref{thm:ruane}.
\end{proof}

\begin{remark}For virtually special actions, $\conc$ follows from \cite[Corollary~9]{ckrw} (since abelian subgroups of right-angled Artin groups, or, more generally, of $GL(n,\mathbb Z)$ are always separable, cf.\ Section~\ref{sec:C}). Thus in the proof of  Corollary~\ref{cor:centralizer} one can replace Theorem~\ref{thm:cubicalpropc} with \cite[Corollary~9]{ckrw}.
\end{remark}

%

\section{Products of symmetric spaces and Euclidean buildings}
\label{sec:buildings}
In this section we discuss how an intersection pattern of flats in a $\czero$ space interacts with $\conc$. The main example is a product of irreducible symmetric spaces of noncompact type and/or irreducible thick Euclidean Tits buildings.

\begin{prop}
	\label{prop:symmetricspace}
	Suppose $G$ acts properly by isometries on $X=X_1\times X_2\times\cdots \times X_n$ such that each $X_i$ is either a nonflat irreducible symmetric space of noncompact type or an irreducible thick Euclidean Tits building with cocompact affine Weyl group. Let $H\le G$ be a finitely generated semisimple virtually abelian subgroup and let $K\le \comm{G}{H}$ be a finitely generated subgroup. Then $K$ normalizes a finite index subgroup of $H$. 
\end{prop}

\begin{proof}
	Recall that the Tits boundary $\partial_T X_i$ is an irreducible spherical building, which has the structure of a simplicial complex. A top--dimensional isometrically embedded sphere in $\partial_T X_i$ is called an \emph{apartment} and $\partial_T X_i$ is a union of apartments. Since we are assuming thickness, each top--dimensional simplex is an intersection of apartments. Let $q_i:X_i\to X_i$ be an isometry. Then $\partial q_i:\partial_T X_i\to \partial_T X_i$ clearly preserves the collections of apartments, and hence it respects the simplicial structure. The Tits boundary $\partial_T X$ is a spherical join of irreducible spherical buildings, and thus it has a structure of a polyhedral complex. \emph{Apartments} in $\partial_TX$ are spherical joins of apartments in each of its factors. Any isometry $q \colon X\to X$ respects the product decomposition (up to permutation of factors), and therefore the induced boundary map $\partial q \colon \partial_T X \to \partial_T X$ respects the polyhedral structure of $\partial_T X$.

	Let $H\le G$ be a finitely generated semisimple virtually abelian subgroup and let $K\le \comm{G}{H}$ be a finitely generated subgroup. By Lemma~\ref{lem:comm invariant} we can assume that $H$ is free abelian. Then $H$ acts on a flat $F\subset X$ cocompactly by translations. Let $S\subset \partial_T X$ be the smallest isometrically embedded sphere containing $\partial_T F$ which is also a subcomplex. Note that at least one such sphere exists, since $\partial_T F$ is contained in an apartment \cite[Proposition 3.9.1]{kleiner1997rigidity}. Let $k\in K$. We claim $\partial k(S)=S$, where $\partial k$ is the boundary map. By Lemma~\ref{lem:split}, $\partial k(\partial_T F)=\partial_T F$. Since $\partial k$ respects the polyhedral structure and $S$ is the smallest spherical subcomplex containing $\partial_T F$, the claim follows.
	The action of $K$ on $\partial_T X$ provides a homomorphism $\beta:K\to \Isom(\partial_T F)$ (note that $\beta$ equals to $\Phi$ from Definition~\ref{def:honolomy}). By the previous claim, each element in $\beta(K)$ is the restriction of an isometry of $S$ which respects the polyhedral complex structure on $S$. Thus $\beta(K)$ is finite, which implies the theorem by Proposition~\ref{prop:normalize}.
\end{proof}

What is really happening in Proposition~\ref{prop:symmetricspace} is that flats in $X$ branch in sufficiently many directions. To this end, we formulate Proposition~\ref{prop:branching} below without referring to the structure of symmetric spaces and Euclidean buildings, where Proposition~\ref{prop:symmetricspace} is a special case of Proposition~\ref{prop:branching}.

\begin{defn}
	\label{def:branching}
	Let $F$ be a flat in a $\czero$ space $X$. Let $\partial_T F$ be the Tits boundary of $F$. A subsphere $S$ of $\partial_T F$ is \emph{singular} if there is a subflat $F_0\subset F$ with $\partial_T F_0=S$ such that the parallel set $P_{F_0}$ of $F_0$ is not contained in a bounded neighborhood of $P_F$. 
\end{defn}

The following generalization of Proposition~\ref{prop:symmetricspace} is straightforward.

\begin{prop}
	\label{prop:branching}
	Suppose $G$ acts on a $\czero$ space $X$ properly by isometries. Let $H\le G$ be a finitely generated free abelian group acting cocompactly on a flat $F\subset X$ by translations. Suppose the collection of all singular subspheres in $\partial_T F$ is rigid in the sense that there are only finitely many isometries of $\partial_T F$ permuting the singular subspheres. Then any finitely generated subgroup $K$ in $\comm{G}{H}$ normalizes a finite index subgroup of $H$.
\end{prop}

\section{Bredon cohomological dimension for virtually abelian stabilizers}\label{sec:bredon}
Let $G$ be a group and let $\mathcal{F}$ be a \emph{family} of subgroups of $G$, i.e.,\ a collection of subgroups which is closed under taking subgroups and conjugation.  Let $\mathrm{cd}_{\mathcal{F}}G$ denote the \emph{Bredon cohomological dimension of $G$ for the family $\mathcal{F}$}. For definition and properties of Bredon cohomological dimension we refer the reader to \cite{lucksurv}. Let us mention that a closely related invariant is the \emph{Bredon geometric dimension} $\mathrm{gd}_{\mathcal{F}}G$ which is the lowest dimension of the universal $G$--CW--complex with stabilizers in~$\mathcal{F}$. These two invariants are related by $\mathrm{cd}_{\mathcal{F}}G \le \mathrm{gd}_{\mathcal{F}}G \le \mathrm{max}\{3, \mathrm{cd}_{\mathcal{F}}G\}$.

For any integer $r \geq 0$, let $\mathcal{F}_r$ denote the family of all subgroups of $G$ which are finitely generated virtually abelian of rank at most $r$. Thus $\mathcal{F}_0$ consists of all finite subgroups of $G$ and $\mathcal{F}_1$ consists of all virtually cyclic subgroups of $G$.
In \cite{cat0vab} there is presented a method for bounding $\mathrm{cd}_{\mathcal{F}_r}G$ for $\czero$ groups, which depends on $\conc$. 

\begin{thm}{\cite[Theorem~1.1]{cat0vab}}\label{thm:cat0vab} Let $G$ be a group acting properly by semisimple isometries on a complete proper $\czero$ space of topological dimension $n$. Suppose additionally that $G$ satisfies $\conc$. Then for any $0 \leqslant r \leqslant n$ we have $\mathrm{cd}_{\mathcal{F}_r}G \leq n+r+1.$
\end{thm}

However, by analyzing the action of $\comm{G}{H}$ on the core $C \subset F^{\perp}$, we are able to remove $\conc$ from the assumptions of the above theorem. We need the following definition.
\begin{defn}
Given a subgroup $H \in \mathcal{F}_r$, let $\mathrm{All}[H]$ denote the family of subgroups of $\comm{G}{H}$ which consists of all subgroups $A$ such that $A \cap H$ is of finite index in $A$. \end{defn}

The only place where $\conc$ is used in the proof of \cite[Theorem~1.1]{cat0vab} is the proof of {\cite[Lemma~3.4]{cat0vab}}, where it is shown that 
\[\mathrm{cd}_{\mathrm{All}[H]} \comm{G}{H} \leq n-r+1.\]

This is obtained in two steps. First, using $\conc$ one writes $\comm{G}{H}$ as the limit $\mathrm{lim}_i N_G(H_i)$ of normalizers of subgroups $H_i$ which are commensurable with $H$. Then one bounds $\mathrm{cd}_{\mathrm{All}[H] \cap N_G(H_i)} N_G(H_i)$ for every $i$ using the proper action of $N_G(H_i)/H_i$ on a $\czero$ space $\minset{H_i} \cap F^{\perp}$.


\begin{prop}\label{prop:bredondimall}Let $G$ be a group acting properly by semisimple isometries on a proper $\czero$ space $X$ of topological dimension $n$. Let $H \in \mathcal{F}_r$ be subgroup of $G$. Then \[\mathrm{cd}_{\mathrm{All}[H]} \comm{G}{H} \leq n-r.\]
\end{prop}

The proposition is an easy consequence of the following theorem of Degrijse-Petrosyan.

\begin{thm}{\cite[Corollary~1]{DePe}} Let $G$ be a group acting by isometries on a separable $\czero$ space of topological dimension $n$ and suppose that the $G$--orbit of every point $x \in X$ is discrete. Let $\mathcal{F}$ be the smallest family of subgroups of $G$ containing the point stabilizers $G_x$ for every $x \in X$. Then we have
  \[\mathrm{cd}_{\mathcal{F}}G  \leq n.\]
\end{thm}

\begin{proof}[Proof of Proposition~\ref{prop:bredondimall}] Consider the action of $\comm{G}{H}$ on the core $\core \subset F^{\perp}$ given by Lemma~\ref{lem:core}. Clearly $\core$ is a $\czero$ space, since it is a convex subset of $X$. We have that $\core$ is separable, since it is a subset of a proper, and hence separable, metric space $X$. Let $\mathrm{dim}$ denote the topological dimension. Notice that \[\mathrm{dim} (F \times F^{\perp}) \leq \mathrm{dim}(X) \leq n.\] Since $\mathrm{dim}(F) =r$ and $\core \subset F^{\perp}$ we obtain that $\mathrm{dim}(\core) \leq n-r$.  By Lemma~\ref{lem:discreteorb} the action has discrete orbits.
 
It remains to check that $\mathrm{All}[H]$ is the smallest family of subgroups of $\comm{G}{H}$ which contains point stabilizers. By definition of $\core$ every point stabilizer is commensurable with $H$ and thus belongs to $\mathrm{All}[H]$. On the other hand, any subgroup $A \in \mathrm{All}[H]$ has a fixed point. To see this, notice that the intersection $A\cap H$ has a fixed point, and since $[A \colon A\cap H]$ is finite, the subgroup $A$ has a finite orbit and thus a fixed point as well.
\end{proof}

Combining proof of \cite[Theorem~1.1]{cat0vab} with Proposition~\ref{prop:bredondimall} we obtain the following.

\begin{thm}\label{thm:improvedcat0vab} Let $G$ be a group acting properly by semisimple isometries on a complete proper $\czero$ space of topological dimension $n$. Then for any $0 \leqslant r \leqslant n$ we have $\mathrm{cd}_{\mathcal{F}_r}G \leq n+r+1.$
\end{thm}

\begin{remark}In Proposition \ref{prop:bredondimall} we obtain a better dimension bound when compared with \cite[Lemma~3.4]{cat0vab}. However, this does not improve the bound for $\mathrm{cd}_{\mathcal{F}_r}G$ in Theorem~\ref{thm:improvedcat0vab}.
Nonetheless, it does simplify the construction, as for any commensurability class $[H]$ one can use a single $\czero$ space $\core \subset F^{\perp}$ rather than a countable collection of spaces $\minset{H_i} \cap F^{\perp}$.
\end{remark}

\section{Relation with Leary-Minasyan groups}\label{sec:lmgroups}
The following construction is due to Leary and Minasyan \cite{LM}.
\begin{defn}
	\label{def:LM}
	Let $T$ be a flat torus of dimension $n$. We identify $H=\pione{T}$ with a subgroup of $\Isom(\mathbb E^n)$. Let $\alpha$ be an element in the commensurator of $H$ in $\Isom(\mathbb E^n)$ such that the induced action of $\alpha$ on the Tits boundary $\partial_T\mathbb E^n$ has infinite order. Let $H_1$ and $H_2$ be finite index subgroups of $H$ such that $\alpha H_1\alpha^{-1}=H_2$. Let $T_1$ and $T_2$ be the coverings of $T$ corresponding to $H_1$ and $H_2$ respectively. Then $\alpha$ descents to an isometry $\alpha':T_1\to T_2$.
	
	Now we define \[X=(T_1\times[0,1])\sqcup T/\sim\] where the relation $\sim$ is defined as follows. We identify points in $T_1\times\{0\}$ with points in $T$ via the covering map $T_1\to T$, and identify points in $T_1\times\{1\}$ with points in $T$ via $T_1\to T_2\to T$ where the first map is $\alpha'$ and the second map is the covering map. We endow $T_1\times[0,1]$ with the product metric. Since the identification maps are local isometric, $X$ has a well-defined quotient metric, and one readily verifies that this metric is locally $\czero$.

	Then $\pione{X}$ is defined to be a \emph{Leary-Minasyan group} (or LM group). Note that $\pione{X}$ is an HNN--extension of form \[\{H,t\mid tH_1t^{-1}=H_2\}\] where the isomorphism between $H_1$ and $H_2$ is induced by $\alpha$. Moreover, $\pione{X}$ is a $\czero$ group acting geometrically on $\widetilde X\cong \mathbb E^n\times T$ where $T$ is a locally finite tree.
\end{defn}

It follows from Proposition~\ref{prop:normalize} that $\conc$ does not hold for LM groups. This leads to the following result of \cite{LM}.

\begin{thm}
	There exists a $\czero$ group which does not satisfy $\conc$.
\end{thm}

\begin{cor}
	\label{cor:LM cor}
	There exists a group $G$ acting geometrically on a $\czero$ piecewise Euclidean complex such that $G$ does not satisfy $\conc$. There exists a closed non-positively curved manifold such that its fundamental group does not satisfy $\conc$.
\end{cor}

\begin{proof}
	For the first statement, we claim that the space $X$ constructed in Definition~\ref{def:LM} admits a piecewise Euclidean structure. To obtain such structure, one chooses an appropriate net inside $X$ and takes the corresponding Voronoi tesselation.
	
	For the second statement, we triangulate $X$ further such that it is a piecewise Euclidean simplicial complex. Then $\pione{X}$ can be embedded as a subgroup of the fundamental group $G'$ of some non-positively curved closed manifold via relative hyperbolization \cite{hu1995retractions}. Clearly $G'$ does not satisfy $\conc$.
\end{proof}

It turns out that LM groups are the only obstructions for a $\czero$ group to satisfy $\conc$ in the following sense.

\begin{prop}
	\label{prop:conc obstruction}
	Let $G$ be a group acting properly on a $\czero$ space $X$ by semisimple isometries. The $\conc$ fails for $G$ if and only if there is a group homomorphism $\eta:G_0\to G$ such that $G_0=\{H,t\mid tH_1t^{-1}=H_2\}$ is a LM group and $\eta|_H$ is injective.
\end{prop}

\begin{proof}
	We first prove the `if' direction. Given the existence of such $\eta$, we claim the subgroup $K\le G$ generated by $\eta(t)$ cannot normalize any finite index subgroup (in particular $K$ is not the trivial subgroup). If the claim does not hold, then there is a finite index subgroup $H'\le H$ such that $\eta(tH't^{-1})=\eta(H')$. Since $tH't^{-1}$ and $H'$ are contained in $H$ and $\eta|_H$ is injective, we get $tH't^{-1}=H'$. This contradicts the definition of LM groups and Proposition~\ref{prop:normalize}.
	
	Now we prove the `only if' direction. Suppose there is an abelian subgroup $H\le G$ such that a finitely generated subgroup $K\le \comm{G}{H}$ does not normalize any finite index subgroups of $H$. Then Proposition~\ref{prop:normalize} implies that $\Phi(K)$ is infinite. Since $\Phi(K)$ is a finitely generated subgroup of some orthogonal group, by Selberg's lemma \cite{selberg1962discontinuous} $\Phi(K)$ has a finite index torsion free subgroup, which is also infinite. Thus there is $t\in K$ such that $\Phi(t)$ is of infinite order. Choose finite index subgroups $H_1,H_2\le H$ such that $tH_1t^{-1}=H_2$. Let $G_0$ be the HNN--extension of $H$ along the isomorphism between $H_1$ and $H_2$ induced by $t$. Clearly there is a homomorphism $G_0\to G$ which is injective on $H$. It remains to show $G_0$ is an LM group. Let $F\subset X$ be a flat where $H$ acts cocompactly. Let $\phi$ be as in Definition~\ref{def:honolomy}. It follows from the proof of Lemma~\ref{lem:comm} that $\phi(t)$ conjugates $H_1$ to $H_2$ when viewing them as subgroups of $\Isom(F)$, which finishes the proof.
\end{proof}

\section{Examples, comments and questions}
\label{sec:examples}

In this section we discuss several examples in the literature which serve as a comparison to results in other sections. The examples show possible pathological behavior of commensurators. The following is a consequence of an example in Wise's thesis \cite{wise}.
\begin{prop}
	\label{prop:anti torus}
	There exists a torsion free group $G$ acting geometrically on a product of two trees such that there is a $\mathbb Z$--subgroup $H\le G$ whose commensurator $\comm{G}{H}$ is not finitely generated. Moreover, $\comm{G}{H}$ does not normalize any finite index subgroup of $H$.
\end{prop}

\begin{proof}
	Let $X$ be the compact non-positively curved square complex defined in \cite[pp.38, Section II.2.1]{wise}. Let $a,b,c,x$ and $y$ be loop in $X$ indicated in \cite[pp.38, Section II.2.1]{wise}. Let $V$ be the subspace of $X$ which is a union of $a,b$ and $c$. Take two copies of $X$ and identify them along $V$ to obtain $X'$ (\cite[Section II.5]{wise}). The universal cover of $X'$ is isomorphic to a product of two trees. We denote edges (which are actually loops) of $X'$ by $a,b,c,x,y,x_1$ and $y_1$. Let $H=\langle c\rangle\le \pione{X'}$ and $K=\comm{\pione{X'}}{H}$. It is clear that $\Phi(K)$ is at most of order two ($\Phi$ is defined in Definition~\ref{def:honolomy}). On the other hand, it follows from \cite[pp.40, Figure 10]{wise} and discussion around there that for any $n>0$, $y^n(y_1)^{-n}\in \comm{\pione{X'}}{H}$ and the biggest subgroup of $H$ normalized by $y^n(y_1)^{-n}$ is of index $2^n$ in $H$. Thus $\comm{\pione{X'}}{H}$ does not normalize any finite index subgroup of $H$, and thus it cannot be finitely generated by Proposition~\ref{prop:normalize}.
\end{proof}

\begin{remark}
	Proposition~\ref{prop:anti torus} shows that the `finitely generated' assumption in Proposition~\ref{prop:normalize}, Theorem~\ref{thm:cubicalpropc} and Proposition~\ref{prop:symmetricspace} cannot be removed.
\end{remark}

Now we discuss another type of irreducible lattices. Let $(p,l)$ be a pair of distinct odd primes and let $\Gam=\Gam_{p,l}$ be the lattice in $PGL_2(\mathbb Q_p)\times PGL_2(\mathbb Q_l)$ defined in \cite[Section 3]{mozes1995actions} and \cite[Chapter 3]{rattaggi2004computations}. It is known that $\Gam_{p,l}$ is torsion free and it acts geometrically on $T_{p+1}\times T_{l+1}$, where $T_k$ denotes the homogeneous tree of degree $k$. The following is a consequence of \cite{RattaggiRobertson}.

\begin{prop}
	\label{prop:highest}
	There exists a torsion free group $\Gam$ acting geometrically on a product of two trees such that there is a $\mathbb Z$--subgroup $H\le \Gam$ which is highest.
\end{prop}

\begin{proof}
	Let $\Gam=\Gam_{p,l}$ be as above. The key property which we need, proven in \cite[Corollary 2.2]{RattaggiRobertson}, is that $\Gam$ is \emph{commutative transitive}. Recall that a group is commutative transitive if the relation of commutativity is transitive on its non-trivial elements. This property implies that if $H$ is a maximal abelian subgroup in the sense that $H$ is not properly contained in another abelian subgroup, then $H$ is highest. However, there exists an example of $\Gam_{p,l}$ which has a maximal abelian subgroup isomorphic to $\mathbb Z$ \cite[Corollary 3.7 and Example 3.8]{RattaggiRobertson}. Thus the proposition follows.
\end{proof}

\begin{cor}
	\label{cor:commensurator computation}
	Let $\Gam=\Gam_{p,l}$ and let $H\le \Gam$ be a non-trivial abelian subgroup. Then the commensurator of $H$ in $\Gam$ is isomorphic to either $\mathbb Z$ or $\mathbb{Z}^2$.
\end{cor}

\begin{proof}
By Theorem~\ref{thm:cubicalpropc}, $H$ satisfies $\conc$. Let $K\le \comm{\Gam}{H}$ be a finitely generated subgroup. Then $K$ normalizes a finite index subgroup $H'\le H$. By Theorem~\ref{thm:combinenormalizerminset}.\ref{thm:centralizer}, the subgroup $\langle K,H'\rangle$ has a finite index subgroup that centralizes $H'$. Thus each $k\in K$ has a non-trivial power which centralizes $H'$. Then the commutative transitivity implies that $K$ is abelian. It follows that $\comm{\Gam}{H}$ is a countable union of abelian subgroups, and thus it is abelian by commutative transitivity. Then the corollary follows.
\end{proof}

\begin{defn}
	\label{defn:regular and singular}
Let $G$ be a group acting geometrically and cellularly on a product of infinite trees $T\times T'$. The Tits boundary of $T\times T'$ is a complete bipartite graph. A $\mathbb Z$--subgroup $H\le G$ is \emph{regular} if the two boundary points of an axis of $H$ are not vertices of the graph $\partial_T (T\times T')$, otherwise $H$ is \emph{singular}. 
\end{defn}

Let $G$ be as in Definition~\ref{defn:regular and singular}. Then $G$ contains both a regular $\mathbb Z$--subgroup and a singular $\mathbb Z$--subgroup \cite[Lemma 8.8]{ballmann1995orbihedra}. Moreover, the commensurator of a regular $\mathbb Z$--subgroup is virtually $\mathbb{Z}^2$ \cite[Lemma 7.13]{ballmann1995orbihedra}.

In the special case where $G$ is $G_{p,l}$ defined above, the commensurator of a singular $\mathbb Z$--subgroup is either $\mathbb Z$ or $\mathbb{Z}^2$. In particular, for a singular $\mathbb Z$--subgroup $H$, the commensurator $\comm{\Gam}{H}$ never acts cocompactly on the parallel set of the axis of $H$. 
This is very different from the case of virtually compact special actions, where the algebraic properties of $\comm{\Gam}{H}$ are always compatible with the geometry of the space on which $\Gam$ acts on. More precisely:

\begin{prop}
	\label{prop:compactible}
	Suppose $W$ is a compact virtually special cube complex. Then for any abelian subgroup $H\le \pione{W}$, the commensurator of $H$ acts cocompactly on the parallel set  $P_F$, where $F$ is a flat in the universal cover $\widetilde W$ stabilized by $H$ such that $H\acts F$ is cocompact.
\end{prop}

\begin{proof}[Sketch of a proof]
Without loss of generality we can assume that $W$ is special. Since there is a local isometric embedding from $W$ to a compact Salvetti complex of some right-angled Artin group, it reduces to proving the lemma in the case where $W$ is a Salvetti complex. In this case, it follows from Servatius' centralizer theorem \cite[Section III]{servatius1989automorphisms} that the centralizer of an abelian subgroup $H$ of a right-angled Artin group acts cocompactly on the parallel set of $F$ in $\widetilde W$, where $F$ is a flat in $\widetilde W$ stabilized by $H$ such that $H\acts F$ cocompact. Now the lemma follows.
\end{proof}

Recall that for a group $G$ acting geometrically and cellularly on a product of two trees, the action is reducible if and only if the quotient cube complex is virtually special \cite{wise}. This together with Corollary~\ref{cor:commensurator computation} and Lemma~\ref{prop:compactible} naturally leads to the following question, which is a variant of a question by Wise on whether irreducible actions always give rise to an anti-torus.

\begin{ques}
	Suppose $G$ acts geometrically and cellularly on a product of two trees $X$. Suppose for each singular $\mathbb Z$--subgroup $H\le G$, the commensurator of $H$ acts cocompactly on the parallel set of an axis of $H$. Is $G$ reducible, i.e.,\ is $G$ commensurable to a product of two free groups?
\end{ques}

We also ask whether one can find examples similar to Proposition~\ref{prop:highest} in the world of symmetric spaces.

\begin{ques}\label{ques:notorious}
Let $G$ be a cocompact lattice in $SL(3,\mathbb R)$. Can $G$ contain a highest $\mathbb Z$--subgroup?
\end{ques}


One readily verifies that if such $\mathbb Z$--subgroup exists, then it is generated by a matrix $M$ in $SL(3,\mathbb R)$ such that
\begin{enumerate}
	\item $M$ has one real eigenvalue and two complex eigenvalues;
	\item the real eigenvalue is not $1$ or $-1$;
	\item the rotation induced by the pair of complex eigenvalues has irrational angle.
\end{enumerate}

However, we do not know whether such matrix can live inside a cocompact lattice, though we speculate that the answer is positive.

\begin{remark} Note that Theorem~\ref{thm:hadamard} and Corollary~\ref{cor:centralizer} have the same conclusion for $\comm{G}{H}$. However, the geometry of the action of $\comm{G}{H}$ on $P_F$ could be quite different. For virtually compact special actions, the action of $\comm{G}{H}$ on $P_F$ is geometric (Proposition~\ref{prop:compactible}). We speculate that this is not the case for geometric actions on Hadamard manifolds. In particular, a positive answer to Question~\ref{ques:notorious} would give an example of a non-cocompact action of $\comm{G}{H}$ on $P_F$ (consider the action on the associated symmetric space). 
\end{remark}
	

\bibliographystyle{alpha}
\bibliography{1}

\end{document}